\documentclass[12 pt]{amsart}

\usepackage{amsthm}
\usepackage{amsmath}
\usepackage{amssymb}
\usepackage{latexsym}
\usepackage{url}
\usepackage{gauss}
\usepackage{enumerate}

\author{Felix Goldberg}
\address{Hamilton Institute, National University of Ireland Maynooth, Ireland}
\email{felix.goldberg@gmail.com}

\title{On monotonicity-preserving perturbations of $M$-matrices}
\date{March 15, 2013}

\newtheorem{thm}{Theorem}[section]
\newtheorem{cor}[thm]{Corollary}
\newtheorem{rmrk}[thm]{Remark}
\newtheorem{lem}[thm]{Lemma}

\newtheorem{defin}[thm]{Definition}

\begin{document}

\begin{abstract}
We obtain an explicit analytical sufficient condition on $E$ that ensures the monotonicity of the matrix $M+E$, where $M$ is an $M$-matrix.
\end{abstract}



\maketitle
 
\section{Introduction} 
A very important concept in applied linear algebra is that of \emph{monotonicity}, going back to Ostrowski \cite{Ost37} and Collatz \cite{Col52}. We say that a matrix $A$ is \emph{monotone} if $A^{-1} \geq 0$ (entrywise). 

From topological considerations it is clear that if $A^{-1}>0$ then all matrices within some open ball around $A$ are also monotone. We are interested in this paper in quantifying this fact. The effect of a small perturbation on the inverse of a matrix may be estimated by 
standard techniques (cf. \cite[Section 5.8]{HornJohnson}) - but the estimates obtained this way are far too weak to be useful.

Recall that a matrix $A$ is an $M$-Matrix if $A=sI-B$, with $B \geq 0$ and $s \geq \rho(B)$. It is well-known (cf. \cite{Avi}) that $M$-matrices are monotone. We would like to study the monotonicity of a perturbation of an $M$-matrix $A$ by a nonnegative matrix $E$. 

Before we usher in our new result (Theorem \ref{thm:main}), we wish to briefly survey the main recent contributions to this problem. For earlier efforts see the introductions to \cite{HuaHayHua12} and to \cite{Bou07}.

\subsection{Buffoni's algorithm}\label{sec:buffoni}
In the 1990 paper \cite{Buf90} (which unfortunately seems to have gone unnoticed until cited for the first time in \cite{HuaHayHua12} in 2012) Buffoni proposed an iterative algorithm that for fixed motonone $A$ and nonnegative $E$ finds the maximum $v>0$ so that $A+vE$ is monotone. We shall call this number $v^{*}$ (it can also be infinite).

Let us denote for $v>0$: $Z(v)=(A+vE)^{-1}$.
\begin{thm}\cite[Theorems 1 and 2]{Buf90}\label{thm:buffoni}
Let $A$ be monotone and $E \geq 0$. Define the following sequence:
$$
v_{k+1}=v_{k}+\min_{i,j,w_{kij}>0}{\frac{z_{kij}}{w_{kij}}}, \ k=0,1,2,\ldots,n, \ , v_{0}=0,
$$
where $Z_{k}=Z(v_{k})=(z_{kij})$ and $W_{k}=-Z^{'}(v_{k})=Z_{k}EZ_{k}=(w_{kij})$.
Then the sequence $\{v_{k}\}$ either converges quadratically to $v^{*}$ if $v^{*}<\infty$ or diverges to $\infty$ otherwise.
\end{thm}

Let us consider an example:
\begin{equation}\label{eq:simple}
A=\left(\begin{array}{ccc}
1.6&0&-0.6\\
-0.4&1.4&0\\
-0.2&-0.4&1.6\\
\end{array}
\right).
\end{equation}

Applying Theorem \ref{thm:buffoni} we get $v^{*}=0.15$ for $E=E_{12}$ and $v^{*}=0.6$ for $E=E_{13}$. These numbers can be verified by elementary algebra. Take now $E=J$ (the all-ones matrix) and then we get $v^{*}=0.0923$, a somewhat less obvious fact.

Buffoni's algorithm has excellent convergence properties and leaves little to be desired when the numerical value of $v^{*}$ is required, as in the small example we've considered. But when analytical bounds are required it is less useful - a point already made by Buffoni himself in \cite{Buf90}.

\subsection{Bouchon's theorem}
A different kind of result was obtained by Bouchon in \cite{Bou07} in 2007. We state it here in terms slightly different from the original. The (directed) graph of a matrix $A$ is defined in the obvious way ($i \rightarrow j$ if and only if $a_{ij} \neq 0$). The distance in this graph will be denoted by $d(\cdot,\cdot)$.

Let $A \in \mathbf{R}^{n}$ and let $\sigma_{i}=\frac{1}{|a_{ii}|}\sum_{j \neq i}{|a_{ij}|}$, and let $J(A)$ be the set of indices $i$ so that $\sigma_{i}<1$. The matrix $A$ is said to be \emph{irreducibly diagonally dominant} if it is irreducible, $J(A) \neq 0$, and $\sigma_{i} \leq 1$ for all $i$. If $\sigma_{i}(A)<1$ for all $i$ we shall say that $A$ is \emph{strictly diagonally dominant}.

We also need to define a number of other quantities associated with $A$ and $E$:
\begin{itemize}
\item
$m(A)=\min_{i=1,\ldots,n}{|a_{ii}|}$
\item
$\eta(A)=\max_{\substack{i,j =1,\ldots,n \\ a_{ij} \neq 0}}{\frac{|a_{ii}|}{|a_{ij}|}}$
\item
$M=M(A,E)=\max_{\substack{i,j =1,\ldots,n \\ e_{ij} \neq 0}}{d(i,j)}$
\item
$C=C(A,E)=\frac{1}{(\eta(A)^{M})Me}.$
\end{itemize}

\begin{thm}\cite[Theorem 2.5]{Bou07}\label{thm:bouchon}
Let $A$ be an irreducibly diagonally dominant $M$-matrix and let $E \in \mathbf{R}^{n}$ so that $E\mathbf{1} \geq 0$. If it holds that $||E||_{\infty} < C \cdot m(A)$, then the matrix $A+E$ is monotone and all its entries are positive.
\end{thm}

Theorem \ref{thm:bouchon} is the first analytical result to give a sufficient bound for monotonicity-preserving perturbations and it is very useful in certain circumstances. However, it has also some limitations. The apparent one - the dependency of $M(A,E)$ on $E$ - is, in fact, not very severe. It is well-known that the diameter of almost every undirected graph is two (cf. \cite[p. 341]{GraphsDigraphs4}), and while for directed graphs the situation is more complicated, in many cases of interest we will actually have $M(A,E)=2$. 

Nevertheless, the bound of Theorem \ref{thm:bouchon} may prove to be very conservative even when $M(A,E)=2$. For the matrix $A$ given in \eqref{eq:simple} and a perturbation $E\geq 0$ all of whose entries are allowed to be non-zero we get: $m(A)=1.4, \eta(A)=4, M=2$ and therefore $C=\frac{1}{32e}$ and $||E||_{\infty} < \frac{1.4}{32e}=0.0161$ as the Bouchon bound in this case. On the other hand, we had seen before (by applying Buffoni's algorithm) that we may actually take \emph{each} entry of $E$ to be as high as $0.0923$ and still preserve monotonicity. 


Our small example highlights another limitation of Theorem \ref{thm:bouchon} - the bound on $E$ is in terms of a norm. While norms have traditionally been the staple of numerical analysis, recently there has been growing recognition of the usefulness of \emph{componentwise} analysis for various applications (cf. \cite{Hig94}).

\subsection{Tridiagonal matrices}
\begin{thm}\cite[Theorem 3.3]{HuaHayHua12}\label{thm:tridi1}
Let 
$$
A=
\left[
\begin{array}{ccccc}
a_{1} & -c_{1} &  &  & \\
-b_{1} & a_{2} & -c_{2} & &  \\
 & \ddots & \ddots & \ddots &\\
 & & -b_{n-2} & a_{n-1} & -c_{n-1}\\
 &  &  & -b_{n-1} & a_{n}
\end{array}
\right]
$$
be a tridigonal $M$-matrix and let $E=hE_{l,k}$, with $h \geq 0$ and $|l-k| \geq 2$. Then $A+E$ is monotone if and only if $h$ satisfies:
$$
h \leq \frac{\prod_{s=l}^{k-1}{c_{s}}}{det M_{[l+1:k-1]}}, \text{if} \ l \leq k-2
$$
or 
$$
h \leq \frac{\prod_{s=k}^{l-1}{b_{s}}}{det M_{[k+1:l-1]}}, \text{if} \ l \leq k+2.
$$
\end{thm}

Theorem \ref{thm:tridi1} gives a very satisfactory closed-form answer to the perturbation problem for the case that $E$ has exactly one nonzero entry. Note that the bound on $h$ is in terms of the entries of $M$.

For the case of a more general fixed nonnegative perturbation $E$ of the tridiagonal $M$-matrix $A$ \cite[Theorem 4.4]{HuaHayHua12} gives an iterative algorithm that determines $v^{*}$ in Buffoni's sense. The authors of \cite{HuaHayHua12} report some experiments that indicate their algorithm may converge even faster than Buffoni's general algorithm.

\subsection{Gavrilov's theorem}
Finally, we wish to mention another elegant but little-known result due to Gavrilov \cite{Gav01}. While not dealing directly with the issue of perturbation, it nevertheless could be potentially very useful in our context. 
\begin{thm}\cite{Gav01}
Let $A$ be a $n \times n$ positive definite matrix. If all the principal submatrices of $A$ of order $m$ for some $2 \leq m<n$ are monotone, then $A$ is monotone.
\end{thm}
Note that the case $m=2$ of Gavrilov's theorem recovers for symmetric matrices the standard fact that $M$-matrices are monotone.

\section{Our new result}

First we recall the famous Sherman-Morrison formula. For its history and context we refer to \cite{HenSea81}.
\begin{lem}\label{lem:sm}
Let $A \in \mathbb{R}^{n \times n}$ be invertible and let $u,v \in \mathbb{R}^{n}$, and $b \in \mathbb{R}$. Then:
$$
(A+buv^{T})^{-1}=A^{-1}-\frac{b}{1+bv^{T}A^{-1}u}A^{-1}uv^{T}A^{-1}.
$$
\end{lem}

We shall also use a monotonicity criterion originally found by Kuttler \cite{Kut71} and stated as condition $N_{42}$ in \cite[p. 137]{Avi}. The notation $x>0$ for a vector $x$ means that all entries of $x$ are nonnegative, with at least one being strictly positive; $x>>0$ means that all entries of $x$ are strictly positive. The inequality $M \geq A$ for matrices is to be understood entrywise. The all-ones vector will be denoted $j$.
\begin{thm}\label{thm:kuttler}
A matrix $A$ is monotone if and only if there exist a monotone matrix $M$ and a vector $w>0$ such that:
\begin{enumerate}[(i)]
\item
$M \geq A$,
\item
$Aw>>0$.
\end{enumerate}
\end{thm}

We now define some quantities associated with $A^{-1}$:
Let $A^{-1}=(b_{ij})$. Let $\Sigma=\sum_{i,j}{b_{ij}}$ and let $\mathbf{r} \in \mathbb{R}^{n}$ and $\mathbf{c} \in \mathbb{R}^{n}$ be the vectors of row and column sums of $A^{-1}$, respectively. 

\begin{defin}
Let the \emph{Buffoni number} of $A$ be:
$$
B_{A}=\min_{i,j}{\frac{b_{ij}}{\mathbf{r}_{i}\mathbf{c}_{j}}}.
$$
\end{defin}

We can now present out main result.
\begin{thm}\label{thm:main}
Let $A \in \mathbb{R}^{n}$ be a strictly diagonally dominant $M$-matrix and let $E \in \mathbb{R}^{n}$. If 
\begin{equation}\label{eq:main}
|E| \leq \frac{B_{A}}{1-B_{A}\Sigma},
\end{equation}
then $A+E$ is monotone. This bound is best possible.
\end{thm}
\begin{proof}
First observe that $Aj>>0$ since $A$ is diagonally dominant and that therefore condition (ii) of Theorem \ref{thm:kuttler} is satisfied for $w=j$. Now let $v=\frac{B_{A}}{1-B_{A}\Sigma}$ and note that $A+E \leq A+vJ$. If we can show that $A+vJ$ is monotone, then we will be done, by Theorem \ref{thm:kuttler}. To see that, we use Lemma \ref{lem:sm}:
\begin{eqnarray}
(A+vJ)^{-1}=A^{-1}-\frac{v}{1+j^{T}A^{-1}j}A^{-1}JA^{-1}=\\=A^{-1}-\frac{v}{1+v\Sigma}\mathbf{r}\mathbf{c}^{T}.
\end{eqnarray}

But $\frac{v}{1+v\Sigma}=B_{A}$ by our choice of $v$ and therefore we see that $A+vJ$ is monotone if and only if $A^{-1} \geq B_{A}\mathbf{r}\mathbf{c}^{T}$ - and that is exactly the definition of $B_{A}$.
\end{proof}

\begin{rmrk}\label{rem:det}
The quantity $\Sigma$ can in fact be easily computed from $A$ by using the following rank-one update formula:
$$
\Sigma=j^{T}A^{-1}j=\frac{|A+J|-|A|}{|A|}.
$$
\end{rmrk}

\section{Examples and discussion}
First let us apply Theorem \ref{thm:main} to the matrix $A$ of \eqref{eq:simple}. Observe that all row and column sums of $A$ equal $1$and so do those of $A^{-1}$. Therefore $\Sigma=3$ in this case and $\mathbf{r}=\mathbf{c}=\mathbf{1}$. Let us compute $A^{-1}$:
$$
A^{-1}=\left(\begin{array}{ccc}
0.6747&0.0723&0.253\\
0.1928&0.7349&0.0723\\
0.1325&0.1928&0.6747\\
\end{array}
\right).
$$
Then $B_{A}$ is equal to the smallest entry of $A^{-1}$, \emph{viz.} $0.0723$. Plugging this value into \eqref{eq:main} we obtain the bound $|E| \leq 0.0923$ - which coincides with the result of Buffoni's algorithm reported in Section \ref{sec:buffoni}.

This example admits a nice generalization. If all the row and column sums of the matrix $A$ are equal to $1$, we shall say that $M$ is \emph{quasi-doubly-stochastic}. In that case $\mathbf{r}=\mathbf{c}=\mathbf{1}$ and $\Sigma=n$ and $B_{A}$ is equal to the smallest entry of $A^{-1}$. Let us state this:
\begin{cor}\label{cor:quasi}
Let $A \in \mathbb{R}^{n \times n}$ be a quasi-doubly-stochastic $M$-matrix. Let $m$ be the smallest entry of $A^{-1}$. Let $E \in \mathbb{R}^{n \times n}$. If 
\begin{equation}\label{eq:q}
|E| \leq \frac{m}{1-mn},
\end{equation}
then $A+E$ is monotone.
\end{cor} 

\subsection{An application to graph Laplacians}
We now turn to another example, which in fact had been the motivating one for this note. It came up in joint work with Steve Kirkland on the subject of eigenvectors of signless Laplacians of graphs. 

Let $A \in \mathbb{R}^{n \times n}$ be the following block matrix whose diagonal blocks are of sizes $s \times s$ and $t \times t$:
$$
A=\left[
        \begin{array}{cc}
           (t+d)I& -J_{s,t}\\
           -J_{t,s}& (s+d)I
        \end{array}
    \right], 
$$
where $s \leq t$ and $d \leq s,t$. We want to find out how big can be a monotonicity-preserving perturbation $E$ without restricting its zero pattern. 

The inverse $A^{-1}$ can be computed using the standard block-matrix inversion formula (cf. \cite[p. 18]{HornJohnson}):
$$
A^{-1}=\left[
        \begin{array}{cc}
           \frac{1}{t+d}I+\frac{t}{d(t+d)(s+t+d)}J& \frac{1}{d(s+t+d)}J\\
           \frac{1}{d(s+t+d)}J& \frac{1}{s+d}I+\frac{s}{d(s+d)(s+t+d)}J
        \end{array}
    \right].
$$
The smallest entry of $A^{-1}$ is $m=\frac{s}{d(s+d)(t+s+d)}$. 
Since the row and column sums of $A$ are all equal to $d$, we see that $\mathbf{r}=\mathbf{c}=d^{-1}\mathbf{1}$ and $\Sigma=nd^{-1}$. Therefore $B_{A}=\frac{ds}{(s+d)(t+s+d)}$ and \eqref{eq:main} gives us (after some algebraic simplifications) the bound
\begin{equation}\label{eq:st_mine}
|E| \leq \frac{s}{d+2s+t}.
\end{equation}

Let us now compute the bound given by Theorem \ref{thm:bouchon}. We have in this case $m(A)=s+d, \eta(A)=t+d$ and $M=2$. Therefore, Bouchon's bound is:
\begin{equation}\label{eq:st_b}
||E||_{\infty} < \frac{s+d}{2e(t+d)^{2}}.
\end{equation}

It is easy to see that in this situation \eqref{eq:st_mine} is much better than \eqref{eq:st_b} - every matrix that satisfies the latter must satisfy the former, but not vice versa. To take a numerical example, for $s=10,t=20,d=5$ we get $|E| \leq 0.2222$ and $||E||_{\infty} < 0.0044$ in the two bounds - a very significant difference.

\subsection{Discussion}
Theorem \ref{thm:main} gives the optimal bound for uniform componentwise perturbation in closed form, without the need to perform Buffoni iterations. However, it requires the knowledge of $A^{-1}$ which might not be possible in some situations. Therefore, we would like to have a way of estimating $B_{A}$ from below in by expressions that depend only on $A$. A possible way to obtain such estimates would be by using the special properties of \emph{inverse $M$-matrices} (cf. \cite{JohSmi11,LiLiuYangLi11}). 

\section{Acknowledgements}
I would like to thank Professor Abraham Berman and Miriam Farber for a careful and sympathetic reading of the paper. Miss Farber has also kindly suggested Remark \ref{rem:det}. I would like to thank the organizers of the 4th International Symposium on Positive Systems (POSTA 2012) that took place in the Hamilton Institute of the National University of Ireland, Maynooth for giving me the opportunity to present an early version of this paper. I would also like to thank A.V. Gavrilov for bringing his paper \cite{Gav01} to my attention.


\section{Acknowledgements}


\bibliographystyle{abbrv}
\bibliography{nuim}

\providecommand{\noopsort}[1]{}
\begin{thebibliography}{10}

\bibitem{Avi}
A.~Berman and R.~J. Plemmons.
\newblock {\em Nonnegative Matrices in the Mathematical Sciences}, volume~9 of
  {\em Classics in Applied Mathematics}.
\newblock SIAM, 1994.

\bibitem{Bou07}
F.~Bouchon.
\newblock Monotonicity of some perturbations of irreducibly diagonally dominant
  {$M$}-matrices.
\newblock {\em Numer. Math.}, 105(4):591--601, 2007.

\bibitem{Buf90}
G.~Buffoni.
\newblock Nonnegative and skew-symmetric perturbations of a matrix with
  positive inverse.
\newblock {\em Math. Comput.}, 54(189):189--194, 1990.

\bibitem{GraphsDigraphs4}
G.~Chartrand and L.~Lesniak.
\newblock {\em Graphs and Digraphs}.
\newblock Chapman and Hall/CRC, 4th edition, 2005.

\bibitem{Col52}
L.~Collatz.
\newblock Aufgaben monotoner {A}rt.
\newblock {\em Arch. Math.}, 3:366--376, 1952.

\bibitem{Gav01}
A.~V. Gavrilov.
\newblock A sufficient condition for the monotonicity of a positive-definite
  matrix. (in {R}ussian).
\newblock {\em Computational Mathematics and Mathematical Physics},
  41(9):1301--1302, 2001.

\bibitem{HenSea81}
H.~Henderson and S.~Searle.
\newblock On deriving the inverse of a sum of matrices.
\newblock {\em SIAM Review}, 23(1):53--60, 1981.

\bibitem{Hig94}
N.~J. Highham.
\newblock A survey of componentwise perturbation theory in numerical linear
  algebra.
\newblock In W.~Gautschi, editor, {\em Mathematics of computation, 1943-1993: a
  half- century of computational mathematics. Mathematics of computation 50th
  anniversary symposium, August 9-13, 1993, Vancouver, Canada.}, volume~48 of
  {\em Proc. Symp. Appl. Math.}, pages 49--77. American Mathematical Society,
  1994.

\bibitem{HornJohnson}
R.~A. Horn and C.~R. Johnson.
\newblock {\em Matrix Analysis}.
\newblock Cambridge University Press, 1985.

\bibitem{HuaHayHua12}
J.~Huang, R.~D. Haynes, and T.-Z. Huang.
\newblock Monotonicity of perturbed tridiagonal {$M$}-matrices.
\newblock {\em SIAM J. Matrix Anal. Appl.}, 33(2):681--700, 2012.

\bibitem{JohSmi11}
C.~R. Johnson and R.~L. Smith.
\newblock Inverse {$M$}-matrices, {II}.
\newblock {\em Linear Algebra Appl.}, 435(5):953--983, 2011.

\bibitem{Kut71}
J.R.Kuttler.
\newblock A fourth-order finite-difference approximation for the fixed membrane
  eigenproblem.
\newblock {\em Math. Comput.}, 25:237--256, 1971.

\bibitem{LiLiuYangLi11}
Y.~Li, X.~Liu, X.~Yang, and C.~Li.
\newblock Some new lower bounds for the minimum eigenvalue of the {H}adamard
  product of an {$M$}-matrix and its inverse.
\newblock {\em Electron. J. Linear Algebra}, 22:630--643, 2011.

\bibitem{Ost37}
A.~M. Ostrowski.
\newblock {\"U}ber die {D}eterminanten mit ?wiegender {H}auptdiagonale.
\newblock {\em Comment. Math. Helv.}, 10:69--96, 1937.

\end{thebibliography}

\end{document}